\documentclass{amsart}
\usepackage{setspace}
\usepackage{a4}
\usepackage{amssymb,amsmath,amsthm,latexsym}
\usepackage{amsfonts}
\usepackage{amsfonts}
\usepackage{graphicx}
\usepackage{textcomp}
\usepackage{cite}
\usepackage{enumerate}
\usepackage[mathscr]{euscript}
\usepackage{mathtools}
\newtheorem{theorem}{Theorem}[section]

\newtheorem{conjecture}[theorem]{Conjecture}

\newtheorem{remark}[theorem]{Remark}

\title{This is the title}
\usepackage{hyperref}
\usepackage{enumerate}
\usepackage{mathtools}
\usepackage{amsmath}
\usepackage{tikz}
\usepackage{hyperref}
\usepackage{graphicx}
\usepackage{textcomp}
\usetikzlibrary{cd}
\usepackage{fancyhdr}

\begin{document}
\begin{center}
{\bf{   C*-ALGEBRAIC CASAS-ALVERO CONJECTURE}\\
K. MAHESH KRISHNA}  \\
Post Doctoral Fellow \\
Statistics and Mathematics Unit\\
Indian Statistical Institute, Bangalore Centre\\
Karnataka 560 059 India\\
Email: kmaheshak@gmail.com\\

Date: \today
\end{center}

\hrule
\vspace{0.5cm}
\textbf{Abstract}: Based on Casas-Alvero conjecture \textit{[J. Algebra, 2001]} we formulate the following conjecture.\\
\textbf{C*-algebraic Casas-Alvero Conjecture : 	Let $\mathcal{A}$ be a  commutative C*-algebra, $n\in \mathbb{N}$ and let $P(z)	\coloneqq (z-a_1)(z-a_2)\cdots (z-a_n)$ be a polynomial over $\mathcal{A}$ with $a_1, a_2, \dots, a_n \in \mathcal{A}$. If $P$ shares a common zero with each of its (first) $n-1$ derivatives, then it is $n^\text{th}$ power of a linear monic C*-algebraic polynomial.}\\
We show that C*-algebraic Casas-Alvero Conjecture holds for C*-algebraic polynomials of degree 2. \\
\textbf{Keywords}: C*-algebra, Casas-Alvero Conjecture.\\
\textbf{Mathematics Subject Classification (2020)}: 46L05, 30C10.\\
\hrule
\tableofcontents
\hrule
\section{Introduction}
In 2001, Prof. Eduardo Casas-Alvero made the following conjecture  \cite{CASASALVERO}. 
\begin{conjecture} \cite{CASASALVERO, DRAISMADEJONG, POLSTRA} \label{CASASALVERO}  \textbf{(Casas-Alvero Conjecture)
		If a complex monic polynomial of degree $n$ shares a common zero with each of its (first) $n-1$ derivatives, then it is $n^\text{th}$ power of a linear monic polynomial.}
\end{conjecture}
Conjecture  \ref{CASASALVERO}  also appears as \textbf{Problem 30} in the list of 33 problems \textbf{Some open problems in low dimensional dynamical systems} by Gasull \cite{GASULL}.

In 2011, using Gauss-Lucas theorem,  Draisma and Jong proved that Conjecture  \ref{CASASALVERO}  holds for polynomials  of degree upto 4  \cite{DRAISMADEJONG}. In 2006, using MAPLE, Diaz–Toca and Gonzalez–Vega  verified Conjecture  \ref{CASASALVERO}  for all polynomials of degree upto 7  \cite{DIAZTOCALAUREANO}.  In 2007, Graf von Bothmer, Labs,  Schicho and van de Woestijne proved that Conjecture  \ref{CASASALVERO}  holds for polynomials  of degree of the form $p^k$ or $2p^k$, where $p$ is a prime and $k$ is a natural number \cite{GRAFVONBOTHMERLABSSCHICHOVANDEWOESTIJNE}. In 2014,  Yakubovich proved that Conjecture  \ref{CASASALVERO}  holds for polynomial with only real roots satisfying some conditions \cite{YAKUBOVICH2016, YAKUBOVICH2014}.  In 2014, Castryck, Laterveer and Ounaies proved that Conjecture  \ref{CASASALVERO}  holds for polynomials  of degree 12  \cite{CASTRYCKLATERVEEROUNAIES}. In 2020, Cima, Gasull and Manosas showed that two conjectures similar to Conjecture  \ref{CASASALVERO}  for smooth functions fail \cite{CIMAGASULMANOSAS}.  In this paper we formulate Conjecture \ref{CASASALVERO} for polynomials over   C*-algebras. We show that  conjecture holds for second degree polynomials over  C*-algebras.
\section{C*-algebraic Casas-Alvero conjecture}
	Let $\mathcal{A}$ be a   C*-algebra. For  $P(z)	\coloneqq (z-a_1)(z-a_2)\cdots (z-a_n)$ for all  $z\in \mathcal{A}$ with  $a_1, a_2, \dots, a_n \in \mathcal{A} $, we define 
\begin{align*}
	P'(z)\coloneqq \sum_{j=1}^{n}(z-a_1)\cdots \widehat{(z-a_j)}\cdots (z-a_n), \quad \forall z \in \mathcal{A}
\end{align*}
where the term with cap is missing.
We can merely formulate the C*-algebraic version of Casas-Alvero conjecture as follows.
\begin{conjecture}\label{CALGEBRAICCASASALVERO}\textbf{(C*-algebraic Casas-Alvero Conjecture)
		Let $\mathcal{A}$ be a  commutative C*-algebra, $n\in \mathbb{N}$ and let $P(z)	\coloneqq (z-a_1)(z-a_2)\cdots (z-a_n)$ be a polynomial over $\mathcal{A}$ with $a_1, a_2, \dots, a_n \in \mathcal{A}$. If $P$ shares a common zero with each of its (first) $n-1$ derivatives, then it is $n^\text{th}$ power of a linear monic C*-algebraic polynomial.}
\end{conjecture}
\begin{theorem}\label{CHOLDS}
	Conjecture \ref{CALGEBRAICCASASALVERO}	holds for C*-algebraic polynomials of degree 2.
\end{theorem}
\begin{proof}
Let $\mathcal{A}$ be a commutative C*-algebra and let $P(z)\coloneqq (z-a)(z-b)$, $a,b \in \mathcal{A}$.	Assume that there is an element  $c\in \mathcal{A}$ such that $P(c)=P'(c)=0$. Then 
	\begin{align*}
		&(c-a)(c-b)=0=2c-a-b \implies c=\frac{a+b}{2} \\
		&\implies \left(\frac{a+b}{2}-a\right)\left(\frac{a+b}{2}-b\right)=0 \implies (b-a)^2=0.
	\end{align*}
By Gelfand-Naimark theorem $a$ and $b$ can be regarded as a complex valued functions. Therefore $a-b=0$ which gives $a=b$. 
\end{proof}

\begin{remark}
	\begin{enumerate}[\upshape(i)]
	\item  \textbf{C*-algebraic Sendov conjecture} has been formulated in \cite{MAHESHKRISHNA}.	
 	\item \textbf{C*-algebraic Schoenberg conjecture} has been formulated in \cite{MAHESHKRISHNA2}.
 	\item \textbf{C*-algebraic Smale mean value conjecture and Dubinin-Sugawa dual mean value conjecture} have been formulated in \cite{MAHESHKRISHNA3}.
	\end{enumerate}
\end{remark}

 \bibliographystyle{plain}
 \bibliography{reference.bib}

\end{document}